\documentclass[12pt]{article}

\usepackage{amsmath,amssymb,mathrsfs,amsthm,xcolor, mathtools}

\newtheorem{Theorem}{Theorem}[section]
\newtheorem{Proposition}[Theorem]{Proposition}
\newtheorem{Lemma}[Theorem]{Lemma}
\newtheorem{Corollary}[Theorem]{Corollary}

\theoremstyle{definition}
\newtheorem{Definition}[Theorem]{Definition}
\newtheorem{Example}[Theorem]{Example}
\newtheorem{Remark}[Theorem]{Remark}

\begin{document}
\title{Volume comparison by timelike Lipschitz maps} 
\author{Hikaru Kubota\thanks{\texttt{u523299e@ecs.osaka-u.ac.jp}\textrm{, Department of Mathematics, Graduate School of Science, University of Osaka, Japan.}}} 
  
\date{\today} 
\maketitle 

\begin{abstract}

In this article, we introduce a modification of the timelike Hausdorff measure $\mathcal{V}^N$ defined by McCann and S\"amann on Lorentzian pre-length spaces. We write the modification of $\mathcal{V}^N$ as $\mathcal{W}^N$. We establish volume comparison inequalities by causality preserving and timelike Lipschitz maps for $\mathcal{V}^N$ and $\mathcal{W}^N$, and discuss basic properties of both $\mathcal{V}^N$ and $\mathcal{W}^N$. Moreover, we show the coincidence of $\mathcal{W}^N$ and $\mathcal{V}^N$ on smooth spacetimes and some Lorentzian pre-length spaces, and construct some examples of timelike Lipschitz maps and causality preserving maps. 
\end{abstract}

\section{Introduction}

 Recently, less regular Lorentz geometry is studied intensively. Such less regular spacetimes naturally appear in general relativity. For example, in \cite{Gal;Gra;Li}, we see a spacetime which exists physically and possesses only $C^0$ regularity. We can find classic results of singularity theorems and these recent advances in low regularity settings in \cite{Ste}. For $C^{1.1}$-metrics, the Hawking singularity theorem, the Penrose singularity theorem, and the Hawking-Penrose singularity theorem were formulated in \cite{Kun;Ste;Vic;HC11}, \cite{Kun;Ste;Vic;PC11}, and \cite{Gra;Gran;Kun;Ste;HPC11}. The first and the second referred above were generalized to $C^1$-metrics in \cite{Gra;TC1}, and the third was generalized in \cite{Kun;Oha;Sch;Ste;HPC1}. We can see basic results for causal structures and geometry of $C^0$-spacetimes in \cite{Lin;C0} and \cite{Chr;Gra;C0} and that there is a unique phenomenon in the $C^0$-case called causal bubble. An important example of causal bubble is found in \cite{Gar;Sou}. In \cite{Gra;Kun;Sa;Ste;future}, the relation of various definitions for causal bubble was studied. As alternative approaches to low regularity Lorentz geometry, closed cone structures, Lorentzian pre-length spaces, and Lorentzian metric spaces were investigated in \cite{Min;cone}, \cite{Kun;Sa;pre-length}, $\cite{Min;Suh1}$.

 In \cite{Kun;Sa;pre-length}, as a synthetic approach to less regular Lorentz geometry, Kunzinger and S\"{a}mann introduced Lorentzian pre-length spaces. A Lorentzian pre-length space is a metric space equipped with the causal and timelike relations, denoted by $\le$ and $\ll$, and a time separation function, denoted by $\tau$. In \cite{Kun;Sa;pre-length}, they generalized causal structures of Lorentzian manifolds to Lorentzian pre-length spaces and defined a condition of timelike curvature bounded below and above by $k\in\mathbb{R}$, denoted by $K\ge(\le)\; k$, as in the case of Alexandrov and CAT spaces. This notion enables us to study less regular Lorentzian geometry in an analogous way to the geometry of metric spaces. For example, for Lorentzian pre-length spaces with $K\ge0$, Beran and S\"{a}mann constructed spaces of directions and tangent cones \cite{Be;Sa;angle}, and Beran, Ohanyan, Rott, and Solis proved a Lorentzian splitting theorem \cite{Be;Oha;Ro;So;splliting}. Moreover, similarly to the geometry of metric measure spaces, we can study the notion of timelike Ricci curvature bounded below and dimension bounded above, called timelike curvature condition. Basic tools and results related to the timelike curvature dimension condition are summarized in \cite{Cav;Mon}, \cite{McC;De}, and \cite{Mon;Su}. See also \cite{Br} and \cite{Br;Oh} for an improved estimate and the Finsler case. In \cite{Mc;Sa}, McCann and S\"{a}mann constructed a Lorentzian analog for Hausdorff measure denoted by $\mathcal{V}^{N}$.

In this work, we introduce a modification $\mathcal{W}^{N}$ of $\mathcal{V}^{N}$. We establish volume comparison inequalities by causality preserving and timelike Lipschitz maps between two Lorentzian pre-length spaces for $\mathcal{V}^{N}$ and $\mathcal{W}^{N}$.

\begin{Definition}[Construction of $\mathcal{W}^{N}$]
Let $(X, d, \le, \ll, \tau)$ be a Lorentzian pre-length space, $A\subseteq X$, $\delta>0$, and  $N\ge0$. We set $\mathcal{W}_{\delta}^{N}(A)$ as
\[\mathcal{W}_{\delta}^{N}(A):=\inf\Biggl{\{} \sum_{i=1}^{\infty} \rho_{N}(J(a_{i},b_{i})) \;\bigg{|}\; J(a_{i},b_{i})\in \mathcal{J}, d(a_{i},b_{i})<\delta \; ,  A\subseteq \bigcup_{i=1}^{\infty} J(a_{i},b_{i})\Biggr{\}},\]
where we set $\mathcal{J} := \{J(p,q)\;|\; p,q\in X\;\mathrm{with}\;p<q\}\cup \{\emptyset\}$, where $J(p,q)=J^{+}(p)\cap J^{-}(q)$ is a causal diamond, and $\rho_{N}(J(p,q)):=\omega_{N}\tau(p,q)^{N}$ with the normalization constant $\omega_{N}\ge1$. Then, we define
\[\mathcal{W}^{N}(A):=\sup_{\delta >0}\mathcal{W}_{\delta}^{N}(A).\]
\end{Definition}
$\linebreak$We can see that $\mathcal{W}^{N}$ becomes a measure on $X$. In the definition of $\mathcal{V}^{N}$, we replace $d(a_{i},b_{i})$ with $\mathrm{diam}_{d}(J(a_{i}, b_{i}))$.

\begin{Definition}[Timelike Lipschitz maps]\label{Def timelike Lipschitz}
Let $X$ and $Y$ be Lorentzian pre-length spaces. Let $\tau_{X}$ and $\tau_{Y}$ be time separations on $X$ and $Y$ and $\lambda>0$. We say that $f:X\to Y$ is a \emph{timelike $\lambda$-Lipschitz map} if 
\[\tau_{Y}(f(p), f(q))\le \lambda\tau_{X}(p, q) \quad\textrm{for all } p,q \in X \textrm{ with } p\ll q\]
holds.
\end{Definition}
Timelike Lipschitz maps were investigated in \cite{Nol} to consider the convergence of spacetimes. Note that for any timelike Lipschitz map $f$, we can see that $f(p) \ll f(q) \Rightarrow p\ll q$ holds.

\begin{Theorem}[Volume comparison inequality for $\mathcal{V}^{N}$]\label{Th Volcom V^n}
Let $(X,d_{X},\le_{X},\ll_{X},\tau_{X})$ and $(Y, d_{Y}, \le_{Y}, \ll_{Y}, \tau_{Y})$ be Lorentzian pre-length spaces. Take $A \subseteq X$. Let $f:X \to Y$ satisfy the following.
\begin{description}
\item{(1)} There exists $\epsilon > 0 $ such that $f$ is uniformly continuous with respect to $d_{X}$ and $d_{Y}$ on $A_ \epsilon$, where  $A_ \epsilon:=\{x\in X| \; d(x, A)\le \epsilon \}$,

\item{(2)} $f$ is timelike $\lambda$-Lipschitz with respect to $\tau_{X}$ and $\tau_{Y}$ on $A_ \epsilon$,

\item{(3)} $f$ is surjective,

\item{(4)} $f$ dually preserves the causal relation, i.e. $p\le q$ $\Leftrightarrow$ $f(p)\le f(q)$ on $X$. 
\end{description}
Then, $\mathcal{V}^{N}(f(A)) \le \lambda^{N} \mathcal{V}^{N}(A)$ holds.
  
\end{Theorem}

\begin{Theorem}[Volume comparison inequality for $\mathcal{W}^{N}$]\label{Th Volcom W^n}
Let $(X,d_{X}, \le_{X}, \ll_{X}, \tau_{X})$ and $(Y, d_{Y}, \le_{Y}, \ll_{Y}, \tau_{Y})$ be Lorentzian pre-length spaces. Take $A \subseteq X$. Let $f:X \to Y$ satisfy the following. 

\begin{description}

\item{(1)} $f$ is uniformly continuous with respect to $d_{X}$ and $d_{Y}$ on $X$,

\item{(2)} $f$ is timelike $\lambda$-Lipschitz with respect to $\tau_{X}$ and $\tau_{Y}$ on $X$,

\item{(3)} $f$ preserves the causal relation, i.e. $p\le q \Rightarrow f(p) \le f(q)$ on $X$.

\end{description}
Then, $\mathcal{W}^{N}(f(A)) \le \lambda^{N} \mathcal{W}^{N}(A)$ holds.
\end{Theorem}
\quad

These inequalities are analogs of a well-known volume comparison inequality for the Hausdorff measure by Lipschitz maps between metric spaces \cite[Proposition 1.7.8(4)]{Bra;Bra;Iva}. As a further research for Lorentzian pre-length spaces with $K\ge0$, we could consider a construction of local coordinates as an analog of the well-known result in the geometry of Alexandrov spaces. Gromov-Hausdorff convergence in the Lorentzian setting is also an important problem and has been investigated intensively in \cite{Nol}, \cite{Kun;Ste;NULLd}, \cite{Min;Suh1}, \cite{Mul} etc. Lorentzian analogs of Hausdorff measure as well as the volume comparison inequalities above will play a fundamental role in those future investigations. 

Furthermore, we see the coincidence of $\mathcal{V}^{N}$ and $\mathcal{W}^{N}$. In smooth spacetimes, through the notion of null distance, we will show $\mathrm{vol}^{g}=\mathcal{V}^{N}=\mathcal{W}^{N}$, where $\mathrm{vol}^{g}$ is the volume measure of the Lorentzian metric and $\mathcal{V}^{N}$ and $\mathcal{W}^{N}$ are with respect to the Riemannian distance of any complete Riemannian metric. Moreover, we will also show $\mathcal{V}^{N} = \mathcal{W}^{N}$ on compact and strongly causal continuous spacetimes and specific Lorentzian pre-length spaces called Lorentzian warped products.  

This paper is organized as follows. We summarize basic definitions and properties of Lorentzian pre-length spaces and timelike Hausdorff measure in Section 2. In Section 3 we prove the volume comparison inequality for $\mathcal{V}^{N}$, introduce $\mathcal{W}^{N}$ as a modification of $\mathcal{V}^{N}$, and prove the volume comparison inequality for $\mathcal{W}^{N}$. Moreover, we show basic properties of $\mathcal{V}^{N}$ and $\mathcal{W}^{N}$. In Section 4, we discuss the coincidence of $\mathcal{V}^{N}$ and $\mathcal{W}^{N}$. In Section 5, we construct examples of timelike Lipschitz maps and causality preserving maps. Section 6 is for outlook.

\section*{Acknowledgements}
The author would like to thank Tobias Beran, Felix Rott, Tadashi Fujioka. Beran and Rott gave him some advices and informed him of null distance in a discussion, and Fujioka gave him advices about the geometry of metric spaces. He is also thankful to his supervisor Shin-ichi Ohta. The author got a lot of supports and advices from him in the whole process of writing this paper, especially in correcting proofs, English grammar, and offering informations of related researches. This work is supported by JST SPRING, Grant Number JPMJST2138.

\section{Lorentzian pre-length spaces and timelike Hausdorff measure}

\subsection{Lorentzian pre-length spaces} 

In this subsection, we review basic definitions of Lorentzian pre-length spaces, causality conditions, and the condition of timelike curvature bound based on the paper \cite{Kun;Sa;pre-length}. We will also see examples and find that Lorentzian pre-length spaces generalize smooth spacetimes and causally plain continuous spacetimes.

\begin{Definition} [Lorentzian pre-length space]\label{Def LpLS}
For a metric space $(X, d)$, let $\le$ be a reflexive and transitive relation on $X$, and let $\ll$ be a transitive relation on $X$. Moreover, we assume that $\ll$ is included in $\le$. Let $\tau:X\times X \to [0, \infty ]$ be a lower semi-continuous function with respect to the topology induced by $d$. Additionally, we assume that $\tau$ satisfies the reverse triangle inequality,

\[\tau(x, z) \ge \tau(x,y) + \tau(y, z) ,\]
for all $x,y,z\in X$ with $x \le y \le z$. Moreover, suppose that $\tau(x, y)=0$ if $x \not\le y$
and $\tau(x, y)>0$ if and only if $x\ll y$. Then we call the quintuplet $(X, d, \le, \ll, \tau)$ a \emph{Lorentzian pre-length space}. We call $\le$ and $\ll$ \emph{causal} and \emph{timelike relations}, respectively, and $\tau$ a \emph{time separation function}.
\end{Definition}

\quad

\begin{Example}[Smooth spacetimes]
A smooth spacetime is a Lorentzian pre-length space with its canonical causal and timelike relations, time separation function, and a metric $d_{h}$ induced from an ambient smooth Riemannian metric $h$.

\end{Example}

\begin{Example}[Lorentzian warped products]\label{Ex Lorentzian warped product}
Let $(X,d)$ be a length metric space, $I \subseteq \mathbb{R}$ be an open interval, and $f:I \to (0, \infty)$ be a continuous function. Let $D$ be the product metric on $Y \coloneq I\times X$. Let $\gamma:[a,b]\to Y$ be a curve such that, by writing $\gamma=(\alpha, \beta)$, $\alpha$ and $\beta$ are both absolutely continuous and $\alpha$ is strictly monotonous. Then, we say that $\gamma$ is 

\[\left\{ \begin{aligned}
&timelike \\
&null \\
&causal \\
 \end{aligned} \right.
  \quad \textrm{if}
  \quad -\dot{\alpha}^{2}+(f\circ\alpha)^{2}\upsilon^{2}_{\beta}
  \quad
   \left\{\begin{aligned}
&< \\
&= \;\;0 \\
&\le \\
 \end{aligned}\right.\]
almost everywhere. Here, $\upsilon_{\beta}$ is the metric derivative of $\beta$. $\gamma$ is called future (past) directed if $\alpha$ is strictly increasing (decreasing). The length $L(\gamma)$ is defined as
\[L(\gamma)=\int^{b}_{a}\sqrt{\dot{\alpha}^{2}-(f\circ\alpha)^{2}\upsilon^{2}_{\beta}} \;dt.\]
We define relations $\le$, $\ll$, and a function $\tau:Y\times Y\to [0, \infty]$ as
\[\left\{\begin{aligned}
&x\le y \Leftrightarrow \textrm{there exists a future directed causal curve from $x$ to $y$}\\
&x\ll y \Leftrightarrow \textrm{there exists a future directed timelike curve from $x$ to $y$}
\end{aligned}\right.\]
and
\[\tau(x,y)= \sup\{L(\gamma)\;|\;\textrm{$\gamma$ is a future directed causal curve from $x$ to $y$}\}.\]
Then, $\le$ is reflexive and transitive, and $\ll$ is transitive. Moreover, from Lemma 3.25 in $\cite{Ale;Gra;Kun;Sa;cone}$, we see that $\tau$ is lower semi-continuous with respect to $D$. Therefore, $(Y, D, \le, \ll, \tau)$ is a Lorentzian pre-length space. We write $Y\equiv I\times_{f}X$ and call it a Lorentzian warped product.

\end{Example}

\begin{Definition}[Future and past]\label{Def FuturePast}
For a Lorentzian pre-length space $X$ and $x\in X$, we define \emph{chronological} and \emph{causal future} of $x$ as
\begin{description}

\item{(1)} $\mathit{I}^+(x) := \{y\in X \;|\; x\ll y\}$,

\item{(2)} $\mathit{J}^+(x) := \{y\in X \;|\; x\le y\}$. 
\end{description}
\end{Definition}
We define \emph{chronological} and \emph{causal past} $\mathit{I}^-(x)$ and $\mathit{J}^-(x)$ by flipping the order of relations above. For $x,y\in X$, write $\mathit{I}(x, y) := \mathit{I}^+(x) \cap \mathit{I}^-(y)$, and we can see that 
$\{\mathit{I}(x, y) \;|\; x,y\in X\}$ is a subbase of a topology on $X$. We call this topology the \textit{Alexandrov topology}.

\begin{Example}[$C^0$ as a Lorentzian pre-length space]
Let $(M,g)$ be a continuous and causally plain spacetime i.e. for every $x\in M$, $\mathrm{int}[J^{+}(x)]\backslash \overline{I^{+}(x)}=\emptyset$. Then, $(M,d_{h},\le,\ll,\tau)$ is a Lorentzian pre-length space, where $d_{h}$ is the distance function induced from a fixed complete Riemannian metric $h$ and
\[\tau(x,y):=\sup\{L_{g}(\gamma) \;|\; \gamma \;\textrm{is a causal curve from \emph{x} to \emph{y}}\}.\]
We call $\mathrm{int}[J^{+}(x)]\backslash \overline{I^{+}(x)}$ a \emph{future bubbling}, and a \emph{past bubbling} is defined as its time dual \cite[Definition 4.1]{Lin;C0}. In the case of continuous spacetimes with non-empty future (past) bubbles, the time separation $\tau$ fails to be a lower semi-continuous function in general \cite[Proposition 5.7]{Kun;Sa;pre-length}.

\end{Example}

\quad

\begin{Definition}[Causal and timelike curves]\label{Def CausalTimelike curve}
Let $\gamma :I \to X$ be a non-constant locally Lipschitz continuous curve with respect to $d$ on an interval $I\subseteq \mathbb{R}$. Then, we say that $\gamma$ is a \emph{future-directed causal (timelike)} curve if for all $t_1, t_2 \in I$ with $t_1 <  t_2$, we have $\gamma (t_1) \le(\ll) \; \gamma (t_2)$. We define \emph{past-directed causal (timelike)} curves by flipping the order of causal (timelike) relation.

\end{Definition}

\begin{Definition}[$\tau$-length]\label{Def tau-length}
Let $\gamma :[a,b] \to X$ be a future directed causal curve. Then, we define the $\tau$-\emph{length} of $\gamma$, denoted by $L_\tau(\gamma)$, as
\[L_\tau(\gamma) := \inf \Bigg{\{} \sum_{i = 0}^{L - 1} {\tau(\gamma (t_i), \gamma (t_{i+1}))}
\;\bigg{|}\; L\in \mathbb{N}, a = t_0 \le t_1 \le \cdots \le t_L = b \Bigg{\}}.\]
\end{Definition} 
$\linebreak$We can see that $\tau$-length satisfies
\[L_{\tau}(\gamma|_{[x, z]}) = L_{\tau}(\gamma|_{[x, y]}) + L_{\tau}(\gamma|_{[y, z]}) \quad(\mathrm{for \;all} \;x\le y\le z \in [a,b]),\]
and $\tau$-length is reparametrization invariant. We say that a causal curve $\gamma:[a,b]\to X$ is \emph{a maximal curve} if $L_{\tau}(\gamma)=\tau(\gamma(a),\gamma(b))$. For a maximal curve $\gamma:[a,b]\to X$, its $\tau$-length is larger than or equal to that of any other causal curve from $\gamma(a)$ to $\gamma(b)$.

As in the case of Lorentz manifolds, we can consider causality conditions on Lorentzian pre-length spaces.

\begin{Definition}[Causality conditions]\label{Def causality condition}
A Lorentzian pre-length space \linebreak $(X, d, \le, \ll, \tau)$ is called 
\begin{description}
\item{(1)} \emph{chronological} if $\ll$ is irreflexive, i.e. $x \not\ll x$ $\mathrm{for \;all}$ $x\in X$,

\item{(2)} \emph{causal} if $\le$ is a partial order, i.e. $x\le y$ and $y\le x$ imply $x = y$ $\mathrm{for \;all}$ $x, y\in X$,

\item{(3)} \emph{strongly causal} if the Alexandrov topology agrees with the topology induced by $d$,

\item{(4)} \emph{globally hyperbolic} if the following two conditions hold:
\begin{description}
\item{(i)} For any compact set $K\subseteq X$, there exists $C>0$ such that $L_{d}(\gamma)\le C$ for any causal curve $\gamma$ in $K$ (non-totally imprisoning),

\item{(ii)} $J^+(x) \cap J^-(y)$ is a compact set $\mathrm{for \;all}$ $x, y\in X$.

\end{description}
\end{description}

\end{Definition}

$\linebreak$
There are other causality conditions, e.g. \emph{locally causally closed}, \emph{causally path connected} and \emph{localizable}. With these conditions, we define Lorentzian length spaces corresponding to length spaces in the geometry of metric spaces (see \cite[Subsection 3.4]{Kun;Sa;pre-length}).  

\begin{Remark}[Causal ladder]
We can see that on Lorentzian length spaces, $(4)\Rightarrow(3)\Rightarrow(2)\Rightarrow(1)$ (\cite[Theorem 3.26]{Kun;Sa;pre-length}).  
\end{Remark}

For Lorentzian pre-length spaces, we can define curvature bound from below or above in the same way as Alexandrov and CAT spaces. For $k\in \mathbb{R}$, we define

\[ M_{k} := \left\{ \begin{aligned}
&\tilde{S}_1^2(r) & &k=\frac1{r}>0 \\
&\mathbb{R}_1^2 & &k=0 \\
&\tilde{H}_1^2(r) & &k=-\frac1{r}<0. \\
\end{aligned} \right.\]
Here, $\tilde{S}_1^2(r)$ and $\tilde{H}_1^2(r)$ are simply connected covering manifolds of two-dimensional Lorentzian pseudosphere (\emph{de Sitter space}), and  two-dimensional Lorentzian pseudo hyperbolic space (\emph{anti-de Sitter space}), respectively. Recall a fact about the existence of triangles of timelike maximizers in $M_{k}$ \cite[Lemma 4.6]{Kun;Sa;pre-length}.

\begin{Lemma}[Realizability]\label{Lem Realizablity}
\quad Let $k\in \mathbb{R}$. Let $a, b, c \in\mathbb{R}^3_{+}$ with $c\ge a+b$. When $c=a+b$ and $k>0$, assume $c<\frac{\pi}{\sqrt{k}}$ additionally. When $c>a+b$ and $k<0$, assume $c<\frac{\pi}{\sqrt{-k}}$ additionally. Then, there exists a timelike geodesic triangle in $M_{k}$ with sidelengths $a, b, c$ uniquely. Here, a timelike geodesic triangle is a triple $(x, y, z)\in M_{k}^3$ with $x\ll y\ll z$ and $\tau(x,z)<\infty$ connected by maximal timelike geodesics, $\alpha$ from $x$ to $y$, $\beta$ from $y$ to $z$, $\gamma$ from $x$ to $z$, satisfying $L_{\tau}(\alpha)=\tau(x, y)$, $L_{\tau}(\beta)=\tau(y, z)$, and $L_{\tau}(\gamma)=\tau(x, z)$.

\end{Lemma}

\begin{Definition}[Curvature bound from below (above) by $k$]\label{Def timelike curvature bound}
\quad A Lorentzian pre-length space $(X,d,\le,\ll,\tau)$ has \emph{timelike curvature bounded from below (above)} by $k\in\mathbb{R}$ if for all $p\in X$, there exists a neighborhood $U_{p}$ of $p$ such that
\begin{description}
\item{(1)} $\tau\vert_{U_{p}\times U_{p}}$ is finite and continuous,

\item{(2)} for all $x, y\in U_{p}$ with $x\ll y$, there exists a future-direct causal curve $\alpha$ from $x$ to $y$ in $U_{p}$ such that $L_{\tau}(\alpha)=\tau(x,y)$,

\item{(3)} let $(x,y,z)$ be a timelike geodesic triangle in $U_{p}$ whose sides, $\alpha$ from $x$ to $y$, $\beta$ from $y$ to $z$, and $\gamma$ from $x$ to $z$, satisfy the timelike size bound for $k$ in $\mathrm{Lemma} \;\ref{Lem Realizablity}$. Let $(\bar{x},\bar{y},\bar{z})$ be a comparison triangle of $(x,y,z)$ in $M_{k}$, realized by timelike geodesics $\bar{\alpha}$, $\bar{\beta}$, and $\bar{\gamma}$. For any two points, $a$ and $b$, on sides of $(x,y,z)$, take points, $\bar{a}$ and $\bar{b}$,  as corresponding points for $a$ and $b$ on sides of $(\bar{x},\bar{y},\bar{z})$. Then, we have $\tau(a,b)\le \bar\tau(\bar{a},\bar{b})$  \;($\tau(a,b)\ge \bar\tau(\bar{a},\bar{b})$), where $\bar\tau$ is the time separation function of $M_{k}$.

\end{description}

\end{Definition}

We call this neighborhood, $U_{p}$, a \emph{comparison neighborhood} of $p$. The equivalence between the timelike curvature bound as above and the corresponding sectional curvature bound in time direction in smooth spacetimes is discussed in \cite{Be;Ku;Oha;Ro}.

\subsection{ McCann and S\"{a}mann's timelike Hausdorff measure} 

In \cite{Mc;Sa}, McCann and S\"{a}mann constructed a Lorentzian analog for Hausdorff measure that we call the \emph{timelike Hausdorff measure}, $\mathcal{V}^{N}$. In this subsection, we recall the definition of timelike Hausdorff measure and its basic properties. In particular, we see that the timelike Hausdorff measure $\mathcal{V}^{N}$ with respect to the Riemannian distance of any complete Riemannian metric coincides with the volume measure $\mathrm{vol}^{g}$ of Lorentzian metric $g$ on strongly causal and causally plain continuous spacetimes with dimension $N$.

\begin{Definition}[Volume of causal diamonds]
Let $N\ge0$. For a Lorentzian pre-length space $(X,d,\le,\ll,\tau)$, set $J(x,y)=J^{+}(x)\cap\; J^{-}(y)$ for $x,y\in X$ with $x\le y$ such that $\tau(x,y)<\infty$, and define 
\[\rho_{N}(J(x,y))\coloneq\omega_{N}\tau(x,y)^{N},\]
where $\omega_{N}\coloneq\frac{\pi^{\frac{N-1}{2}}}{N\Gamma(\frac{N+1}{2})2^{N-1}}$ and $\Gamma(x):=\int_0^{\infty} t^{x-1}e^{-t}dt$ is the Euler's gamma function.

\end{Definition}

\begin{Remark}
We can see that for an integer $N\ge2$ and $x,y\in X$ with $\tau(x,y)<\infty$, 
\[\rho_{N}(J(x,y))=\mathrm{vol}^{\mathbb{R}_{1}^{N}}(\tilde{J}(\tilde{x},\tilde{y})),\]
where $\tilde{x}$ and $\tilde{y}$ are points in the $N$-dimensional Minkowski space $\mathbb{R}_{1}^{N}$ such that $\tilde{\tau}(\tilde{x},\tilde{y})=\tau(x,y)$. Recall that the above equation holds independently from the choice of $\tilde{x}$ and $\tilde{y}$.

\end{Remark}

\begin{Definition}[Timelike Hausdorff measure $\mathcal{V}^{N}$]\label{Def V^n}
Let $(X,d,\le,\ll,\tau)$ be a Lorentzian pre-length space and $N\ge0$. Set $\mathcal{J} \coloneq \{J(x,y)\;|\; x,y\in X,\;x<y\}\cup \{\emptyset\}$. By setting $\rho_{N}(\emptyset):=0$ and $\rho_{N}(J(x,y)):=\infty$ when $\tau(x,y)=\infty$, we extend $\rho_{N}: \mathcal{J} \to [0,\infty ]$. For $\delta >0$ and $A \subseteq X$, we define

\[\mathcal{V}_{\delta}^{N}(A):=\inf\Bigg{\{} \sum_{i=1}^{\infty} \rho_{N}(A_{i}) \;\bigg{|}\; A_{i}\in \mathcal{J}, \mathrm{diam}_{d}(A_{i})<\delta,  A\subseteq \bigcup_{i=1}^{\infty} A_{i}\Bigg{\}},\]
with $\inf\emptyset =\infty$. Moreover, we define

\[\mathcal{V}^{N}(A):=\sup_{\delta>0} \mathcal{V}_{\delta}^{N}(A).\]
\end{Definition}

We can see that $\mathcal{V}^{N}$ is a Borel measure on $X$ similarly to the case of Hausdorff measures on metric spaces.

\begin{Remark} 
We can restrict countable covers in Definition \ref{Def V^n} to finite coverings \cite[Proposition 3.8]{Mc;Sa}. In fact set a Borel outer measure $\tilde{\mathcal{V}}_{\delta}^{N}$ on $X$ as
\[\tilde{\mathcal{V}}_{\delta}^{N}(A):=\inf\Bigg{\{} \sum_{i=1}^{L} \rho_{N}(A_{i}) \;\bigg{|}\; A_{i}\in \mathcal{J}, \mathrm{diam}(A_{i})<\delta,  A\subseteq \bigcup_{i=1}^{L} A_{i}\Bigg{\}},\]
then we have $\tilde{\mathcal{V}}_{\delta}^{N}=\mathcal{V}_{\delta}^{N}$ for all $\delta$. Therefore, $\tilde{\mathcal{V}}^{N}=\mathcal{V}^{N}$ follows.

\end{Remark}

\begin{Definition}[Timelike Hausdorff dimension]\label{Def Timelike Hausdorff dimension}
Let $(X,d,\le,\ll,\tau)$ be a Lorentian pre-length space and $\{\mathcal{V}^{N}\}_{N\ge 0}$ be the family of timelike Hausdorff measures. For $B\subseteq X$, we define its \emph{timelike Hausdorff dimension}, denoted by  $\dim^{\tau}(B)$, as

\[\dim^{\tau}(B):=\inf\{ N \ge0 \;|\; \mathcal{V}^{N}(B)<\infty \},\]
with $\inf\emptyset =\infty$.

\end{Definition}

With an additional assumption, \emph{local d-uniformity}, we see that the timelike Hausdorff dimension of $B\subseteq X$ exists uniquely.

\begin{Definition}[Local $d$-uniformity]
For a Lorentzian pre-length space $(X,d,\le,\ll,\tau)$, we say that $X$ is \emph{locally $d$-uniform} if every point $p\in X$ has a neighborhood, $U_{p}$, such that $\tau(x,y)=o(1)$ on $U_{p}$ as $d(x,y)\to 0$.

\end{Definition}

In \cite[Lemmas 3.3, 3.5]{Mc;Sa} and \cite[Theorem 4.8]{Mc;Sa}, we see the uniqueness of the timelike Hausdorff dimension and the coincidence of $\mathcal{V}^{N}$ and $\mathrm{vol}^{g}$ on continuous causally plain and strongly causal spacetimes of dimension $N$.
\begin{Theorem}[Uniqueness of dimension]\label{Th uniqueness of dimension v^n}
Let $(X,d,\le,\ll,\tau)$ be a Lorentzian pre-length space, $A\subseteq B\subseteq X$ and $0\le k_1 <\dim^{\tau}(B)<k_2<\infty$. Then, we have

\begin{description}
\item{(1)} $\dim^{\tau}(A)\le \dim^{\tau}(B)$,

\item{(2)} $\mathcal{V}^{k_1}(B)=\infty$,

\item{(3)} if $(X,d,\le,\ll,\tau)$ is locally $d$-uniform, then $\mathcal{V}^{k_2}(B)=0$.

\end{description}

\end{Theorem}

\begin{Theorem}
  Let $(X,d,\le,\ll,\tau)$ be a locally $d$-uniform Lorentzian pre-length space, and $X=\bigcup_{i\in\mathbb{N}}U_{i}$. Then,
  \[\dim^{\tau}(X)=\sup_{i\in\mathbb{N}}\dim^{\tau}(U_{i}).\]
\end{Theorem}

\begin{Theorem}[$\mathrm{vol}^{g}=\mathcal{V}^{N}$]\label{Th Volg=V^n}
Let $(M,g)$ be a strongly causal, continuous and causally plain spacetime of dimension $N$ and $\mathrm{vol}^{g}$ be the volume measure of the Lorentzian metric $g$ on $M$. Then, $\mathrm{vol}^{g}=\mathcal{V}^{N}$ holds. Here, $\mathcal{V}^{N}$ is with respect to the Riemannian distance of any complete Riemannian metric on $M$.

\end{Theorem}

\section{Volume comparison inequality} 

In the geometry of metric spaces, volume comparison by Lipschitz maps is a fundamental yet important fact \cite[Proposition 1.7.8 (4)]{Bra;Bra;Iva}. It plays a role to show that all open sets have the same Hausdorff dimension for Alexandrov spaces, and that is a necessary step to construct local coordinates on them. In addition,  the volume comparison is a fundamental tool in the convergence theory of Riemannian manifolds. Theorem \ref{Th Volcom V^n} and Theorem \ref{Th Volcom W^n} can be regarded as Lorentzian analogs of that volume comparison.

\subsection{Volume comparison inequality for $\mathcal{V}^{N}$}

First, we establish the volume comparison inequality between Lorentzian pre-length spaces by a timelike Lipschitz map for $\mathcal{V}^{N}$.

$\linebreak$\textbf{Theorem \ref{Th Volcom V^n}} \textit{Let $(X, d_{X}, \le_{X}, \ll_{X}, \tau_{X})$ and $(Y, d_{Y}, \le_{Y}, \ll_{Y}, \tau_{Y})$ be Lorentzian pre-length spaces. Take $A \subset X$. Let $f:X \to Y$ satisfy the following.}
\begin{description}
\item{\textit{(1)}} \textit{There exists $\epsilon > 0 $ such that $f$ is uniformly continuous  with respect to $d_{X}$ and $d_{Y}$ on $A_{\epsilon}$, where  $A_ \epsilon:=\{x\in X| \; d_{X}(x, A)\le \epsilon \}$,}

\item{\textit{(2)}} \textit{$f$ is $\lambda$-timelike Lipschitz with respect to $\tau_{X}$ and $\tau_{Y}$ on $A_ \epsilon$,}

\item{\textit{(3)}} \textit{$f$ is surjective,}

\item{\textit{(4)}} \textit{$f$ dually preserves the causal relation, i.e. $p\le q$ $\Leftrightarrow$ $f(p)\le f(q)$ on $X$.}

\end{description}
\textit{Then, $\mathcal{V}^{N}(f(A)) \le \lambda^{N} \mathcal{V}^{N}(A)$ holds.}

\begin{proof} Take $\epsilon >0$ that satisfies the conditions $(1)$ and $(2)$. For $k\in \mathbb{N}$, take $D(k)\in (0,\frac{1}{k})$ such that for any $p,q\in A_{\epsilon}$ with $d_{X}(p,q)<D(k)$, $d_{Y}(f(p),f(q))<\frac{1}{k}$ holds. Take $k >\frac{1}{\epsilon}$ and any covering $\{J(a_{i}, b_{i})\}_{i=1}^{L}$ of $A$ such that
\[\left\{ \begin{aligned} &J(a_{i}, b_{i})\cap A \not= \emptyset, \\
                          &\mathrm{diam}_{d_{X}}(J(a_{i}, b_{i}))<D(k)<\epsilon. \\
             \end{aligned} 
             \right.
             \]
Then, $J(a_{i},b_{i})\subset A_{\epsilon}$, and we can see from $(3)$ and $(4)$ that $f(J(a_{i}, b_{i}))=J(f(a_{i}), f(b_{i}))$ holds. Indeed, since $x\le y$ implies $f(x)\le f(y)$, $f(J(a_{i}, b_{i}))\subseteq J(f(a_{i}), f(b_{i}))$ holds. To get the reverse inclusion, take any $y\in J(f(a_{i}), f(b_{i}))$. Since $f$ is surjective and is dually preserving the causal relation, there exists $x\in X$ such that $f(x)=y$ and $x\in J(a_{i}, b_{i})$ hold. Then, it follows that $y\in f(J(a_{i}, b_{i}))$, and we have $J(f(a_{i}), f(b_{i}))\subseteq f(J(a_{i}, b_{i}))$. 

Hence, $\{J(f(a_{i}), f(b_{i}))\}_{i=1}^{L}$ is a finite covering of $f(A)$ and satisfies

\[\left\{ \begin{aligned} &J(f(a_{i}), f(b_{i}))\cap f(A) \not= \emptyset, \\
                          &\mathrm{diam}_{d_{Y}}(J(f(a_{i}), f(b_{i})))<\frac{1}{k}\;. \\
             \end{aligned} 
             \right.
             \]
Then, we have
 \[\begin{aligned}
 \mathcal{V}_{\frac{1}{k}}^{N}(f(A))\le \sum_{i=1}^{L} \rho_{N}(J(f(a_{i}), f(b_{i}))
                                    &=\sum_{i=1}^{L} \omega_{N}\tau_{Y}(f(a_{i}),f(b_{i}))^{N}\\
                                    &\le \sum_{i=1}^{L} \omega_{N}\lambda^{N}\tau_{X}
                                    (a_{i},b_{i})^{N}.
  \end{aligned}\]
Taking the infimum over all $\{J(a_{i}, b_{i})\}_{i=1}^{L}$ as above, we have
\[\mathcal{V}_{\frac{1}{k}}^{N}(f(A)) \le \lambda^{N} \mathcal{V}_{D(k)}^{N}(A),\]
and letting $k \to \infty$, we have $\mathcal{V}^{N}(f(A)) \le \lambda^{N} \mathcal{V}^{N}(A)$.
\end{proof}

\begin{Corollary}
Let $X$ and $Y$ be Lorentzian pre-length spaces and $f:X\to Y$ satisfy the conditions in Theorem \ref{Th Volcom V^n}. Then
\[\dim^{\tau}(f(A))\le \dim^{\tau}(A)\]
holds.
\end{Corollary}

\begin{Corollary}
Assume that $\tau_{Y}$ is uniformly continuous on $Y$ with respect to $d_{Y}$. Then, we can weaken the condition (3) in Theorem \ref{Th Volcom V^n} to
\[f(X)\supset \bigcup \bigg{\{}J(x,y) \;\bigg{|}\; J(x,y)\cap f(A) \not= \emptyset, \tau_{Y}(x,y)\le L\bigg{\}}\]
for some $L>0$.
\end{Corollary}

\begin{proof}
Since $f$ and $\tau_{Y}$ are uniformly continuous on $X$ and $Y$ with respect to $d_{X}$ and $d_{Y}$, taking $k$ in the proof of Theorem \ref{Th Volcom V^n} sufficiently large, we have $\tau_{Y}(f(x),f(y))<L$ for all $x,y\in X$ such that $d_{X}(x,y)<D(k)$.
\end{proof}

$\linebreak$Let $X$ be a Lorenzian pre-length space and $A$ be a subset of $X$. With the restriction of the distance function, the causal relations, and the time separation of $X$, $A$ is a Lorentzian pre-length space as well. Set the timelike Hausdorff measure on $A$ with restrictions as $\mathcal{V'}^{N}$. Then, as opposed to the case of the Hausdorff measure on metric spaces, we can see $\mathcal{V}^{N}(A)\not= \mathcal{V'}^{N}(A)$ in general. For example, take the spacelike line segment $l=\{t=0, 0\le x\le1\}$ in the 2-dimensional Minkowski space $\mathbb{R}^{2}_{1}$ and we have  $\mathcal{V}^{1}(l)\le1$ and $\mathcal{V'}^{1}(l)=\infty$. Nonetheless we can see that for any relatively compact open subset $A$ in a strongly causal and causally path connected Lorentzian pre-length space $X$, $\mathcal{V}^{N}(A)= \mathcal{V'}^{N}(A)$ holds.

\begin{Definition}\cite[Definition 3.1]{Kun;Sa;pre-length}
  A Lorentzian pre-length space is \emph{causally path connected} if any $x,y\in X$ with $x\ll y$ $(x<y)$ are connected by a timelike (causal) curve.
\end{Definition}

\begin{Proposition}\label{Pro V^n(A)=V'^N(A)}
  Let $(X, d, \le, \ll, \tau)$ be a Lorentzian pre-length space and $A\subseteq X$ be open. Then, $\mathcal{V'}^{N}(A)\le\mathcal{V}^{N}(A)$ holds. Moreover, if $A$ is relatively compact and $X$ is strongly causal and causally path connected, $\mathcal{V}^{N}(A)\le\mathcal{V'}^{N}(A)$ holds. Here, $\mathcal{V'}^{N}$ is the timelike Hausdorff measure on $(A, d\vert_A, {\le}\vert_A, {\ll}\vert_A, \tau\vert_A)$.
\end{Proposition}

\begin{proof}
First, we show that $\mathcal{V'}^{N}(A)\le\mathcal{V}^{N}(A)$ holds. We take a sequence of closed subsets $\{D_{n}\}_{n\in \mathbb{N}}\subset A$ as
\[D_{n}=\Bigg{\{}a\in A\;\bigg{|}\; d(a, A^{c})\ge \frac{1}{n}\Bigg{\}}.\]
Since $D_{n}\uparrow A$, we see that it is sufficient for $\mathcal{V'}^{N}(A)\le\mathcal{V}^{N}(A)$ to show
\[\mathcal{V'}^{N}(D_{n})\le\mathcal{V}^{N}(D_{n}).\]
Let $n\in\mathbb{N}$ and $0<\delta<\frac{1}{2n}$. Take any covering $\{J(a_{i}, b_{i})\}_{i\in \mathbb{N}}$ of $D_{n}$ which satisfies that $\mathrm{diam}_{d}(J(a_{i}, b_{i}))<\delta$ and $J(a_{i}, b_{i})\cap D_{n}\not=\emptyset$. Since $d(a_{i}, D_{n})<\frac{1}{2n}$ and $d(b_{i}, D_{n})<\frac{1}{2n}$, we have $a_{i}, b_{i}\in A$, and $\mathrm{diam}_{d}(J(a_{i}, b_{i})\cap A) \le\mathrm{diam}_{d}(J(a_{i}, b_{i}))<\delta$. Therefore, with $\tau(a_{i}, b_{i})=\tau\vert_A(a_{i}, b_{i})$, we have
\[\mathcal{V'}^{N}_{\delta}(D_{n})\le\mathcal{V}^{N}_{\delta}(D_{n}).\]
Then, letting $\delta\to 0$, $\mathcal{V'}^{N}(D_{n})\le\mathcal{V}^{N}(D_{n})$ follows. We can see that this inequality holds when $\mathcal{V'}^{N}(A)=\infty$.

Next, we assume that $X$ is strongly causal and causally path connected and $A\subseteq X$ is relatively compact. To get $\mathcal{V}^{N}(A)\le\mathcal{V'}^{N}(A)$, for $p\in A$ and $n\in \mathbb{N}$, we set $B_{p,n}\coloneq B_{d}(p,\frac{1}{n})$ and $\beta_{p,n}\subseteq B_{p,n}$ as an open neighborhood of $p$ such that any causal curve $\gamma:[a,b]\to X$ with $\gamma(a),\gamma(b)\in \beta_{p,n}$ satisfies $\gamma([a,b])\subset B_{p,n}$. 

Let $n\in\mathbb{N}$ be sufficiently large. Then, we take a finite covering $\{\beta_{p_{k},6n}\}_{k=1}^{L}$ of $D_{2n}$ and the Lebesgue number $\delta_{n}\in(0,\frac{1}{2n})$ of the covering. Take a causal diamond $J(a, b)$ such that $a, b\in A$, $J(a,b)\cap D_{n}\not=\emptyset$ and $\mathrm{diam}_{d}(J(a,b)\cap A)<\delta_{n}$. Then, we have $a,b\in D_{2n}$ and in particular $a,b\in\beta_{p_{k},6n}$ for some $k$. Then, since $X$ is causally path connected, for any $q\in J(a,b)$ we can take future-directed causal curves $\gamma_{1}:[0,1]\to X$ from $a$ to $q$ and $\gamma_{2}:[0,1]\to X$ from $q$ to $b$. By connecting $\gamma_{1}$ and $\gamma_{2}$, we can see that there is a future-directed causal curve from $a$ to $b$ which includes $q$. In particular, we have $q\in B_{p_{k}, 6n}$. Therefore, we have $J(a,b)\subseteq B_{p_{k}, 6n}$ and $\mathrm{diam}_{d}(J(a,b))<\frac{1}{3n}$ holds. Since $J(a,b)\cap D_{n}\not=\emptyset$, $J(a, b)\subset A$ holds, and we have $\mathrm{diam}_{d}(J(a,b)\cap A) =\mathrm{diam}_{d}(J(a,b))$.

Let $\epsilon\in(0,\delta_{n})$. Take any covering $\{J(a_{i},b_{i})\}_{i\in\mathbb{N}}$ of $D_{n}$ such that $a_{i},b_{i}\in A$, $J(a_{i},b_{i})\cap D_{n}\not=\emptyset$, and $\mathrm{diam}_{d}(J(a_{i},b_{i})\cap A)<\epsilon$. Then, with $\mathrm{diam}_{d}(J(a_{i},b_{i})\cap A) =\mathrm{diam}_{d}(J(a_{i},b_{i}))$ and $\tau(a_{i}, b_{i})=\tau\vert_A(a_{i}, b_{i})$, we have
\[\mathcal{V}^{N}_{\epsilon}(D_{n})\le\mathcal{V'}^{N}_{\epsilon}(D_{n}),\]
and letting $\epsilon\to 0$, we have $\mathcal{V}^{N}(D_{n})\le\mathcal{V'}^{N}(D_{n}).$
\end{proof}

\subsection{A modification of $\mathcal{V}^{N}$}

The condition $(3)$ in Theorem \ref{Th Volcom V^n} is not desirable to construct volume comparison maps. To remove the condition, we modify the definition of timelike Hausdorff measure $\mathcal{V}^{N}$ by replacing the diameter of causal diamonds with the distance of their vertices. In this subsection, we see the basic properties of $\mathcal{W}^{N}$.

\begin{Definition}\label{Def W^n}
For a Lorentzian pre-length space $X$, $A\subseteq X$, and $\delta>0$, we set $\mathcal{W}_{\delta}^{N}(A)$ as
  \[\mathcal{W}_{\delta}^{N}(A):=\inf\Bigg{\{} \sum_{i=1}^{\infty} \rho_{N}(J(a_{i},b_{i})) \;\bigg{|}\; J(a_{i},b_{i})\in \mathcal{J}, d(a_{i},b_{i})<\delta,  A\subseteq \bigcup_{i=1}^{\infty} A_{i}\Bigg{\}},\]
with $\inf\emptyset =\infty$. Moreover, we define

\[\mathcal{W}^{N}(A):=\sup_{\delta>0} \mathcal{W}_{\delta}^{N}(A).\]
Here, $\rho_{N}$ and $\mathcal{J}$ are defined as in Definition \ref{Def V^n}.

\end{Definition}

We see that $\mathcal{W}^{N}$ is a measure on $X$. From the definitions of $\mathcal{V}^{N}$ and $\mathcal{W}^{N}$, we have $\mathcal{W}^{N}\le\mathcal{V}^{N}$ for $N\ge0$. Similarly to $\mathcal{V}^{N}$, we can reduce countable coverings in Definition \ref{Def W^n} to finite coverings. Moreover, we can define timelike 
Hausdorff dimension for $\mathcal{W}^{N}$, denoted by $\mathrm{Dim}^{\tau}$. Since $\mathcal{W}^{N}\le\mathcal{V}^{N}$, we have $\mathrm{Dim}^{\tau}(A)\le \mathrm{dim}^{\tau}(A)$ for any $A\subseteq X$. 

Let $A\subseteq X$, $\epsilon, \delta>0$, and $J(x,y)$ be any causal diamond such that $J(x, y)\cap A\not=\emptyset$ and $d(x,y)<\delta$. Then, in general, $x$ and $y$ are not necessarily in $A_{\epsilon}$ for any sufficiently small $\delta>0$. Therefore, for $\mathcal{W}^{N}$, the uniqueness of dimension does not hold generally even on locally $d$-uniform Lorentzian pre-length spaces. For Lorentzian pre-length spaces with uniformly continuous time separation, we have the uniqueness of timelike Hausdorff dimension for $\mathcal{W}^{N}$.

\begin{Theorem}\label{Th timelike dimension W^n}
  Let $(X,d,\le,\ll,\tau)$ be a Lorentzian pre-length space, $A\subseteq B\subseteq X$ and $0\le k_1 <\mathrm{Dim}^{\tau}(B)<k_2<\infty$. Then, we have

\begin{description}
\item{(1)} $\mathrm{Dim}^{\tau}(A)\le \mathrm{Dim}^{\tau}(B)$,

\item{(2)} $\mathcal{W}^{k_1}(B)=\infty$,

\item{(3)} if $\tau$ is uniformly continuous, then $\mathcal{W}^{k_2}(B)=0$ holds.

\end{description}

\end{Theorem}

\begin{proof}
$(1)$ and $(2)$ follow similarly to the proof of $(i)$ and $(ii)$ in \cite[Lemma 3.3]{Mc;Sa}. Let $B\subseteq X$ and $ k_{3}\in[\mathrm{Dim}^{\tau}(B),k_{2})$ such that $\mathcal{W}^{k_{3}}(B)=G<\infty$. Take $\delta>0$ and a covering $\{J(a_{i}, b_{i})\}_{i\in \mathbb{N}}$ of $B$ such that $d(a_{i}, b_{i})<\delta$ and $\sum_{i}\rho_{k_{3}}(J(a_{i},b_{i}))<G+1$. Then, we have that
\[
\begin{aligned}
\mathcal{W}^{k_{2}}(B) & \le \sum_{i\in\mathbb{N}}\rho_{k_{2}}(J(a_{i},b_{i})) \\
                       & = \frac{\omega_{k_{2}}}{\omega_{k_{3}}} \sum_{i\in\mathbb{N}} \rho_{k_{3}}(J(a_{i},b_{i}))\tau(a_{i},b_{i})^{k_{2}-k_{3}} \\
                       & <\frac{\omega_{k_{2}}}{\omega_{k_{3}}}(G+1) \;o(1)^{k_{2}-k_{3}} \to 0,
\end{aligned}
\]
as $\delta \to 0$. Therefore, we have $\mathcal{W}^{k_{2}}(B)=0$.
\end{proof}

\begin{Theorem}
  Let $(X,d,\le,\ll,\tau)$ be a Lorentzian pre-length space, with uniformly continuous time separation $\tau$. Let $X=\bigcup_{i\in\mathbb{N}}U_{i}$. Then, 
  \[\mathrm{Dim}^{\tau}(X)=\sup_{i\in\mathbb{N}}\mathrm{Dim}^{\tau}(U_{i}).\]
\end{Theorem} 

\begin{proof}
  The proof goes exactly in the same way as that of Lemma 3.5 in $\cite{Mc;Sa}$.
\end{proof}

\subsection{Volume comparison inequality for $\mathcal{W}^{N}$}
In this subsection, we establish the volume comparison inequality for $\mathcal{W}^{N}$, and we see that for $\mathcal{W}^{N}$, we do not need to require the map $f$ to be surjective and to preserve the causal relation dually.

$\linebreak$\textbf{Theorem \ref{Th Volcom W^n}} \textit{Let $(X, d_{X}, \le_{X}, \ll_{X}, \tau_{X})$ and $(Y, d_{Y} ,\le_{Y}, \ll_{Y}, \tau_{Y})$ be Lorentzian pre-length spaces. Take $A \subseteq X$. Let $f:X \to Y$ satisfy the following.}

\begin{description}

\item{\textit{(1)}} \textit{$f$ is uniformly continuous with respect to $d_{X}$ and $d_{Y}$ on $X$,}

\item{\textit{(2)}} \textit{$f$ is timelike $\lambda$-Lipschitz with respect to $\tau_{X}$ and $\tau_{Y}$ on $X$,}

\item{\textit{(3)}} \textit{$f$ preserves causality i.e. $x \le y \Rightarrow f(x) \le f(y)$ on $X$}.

\end{description}$\linebreak$
\textit{Then, $\mathcal{W}^{N}(f(A)) \le \lambda^{N} \mathcal{W}^{N}(A)$ holds.}

\begin{proof}
For $k\in \mathbb{N}$, set $D(k)$ as in Theorem \ref{Th Volcom V^n}. Take any covering $\{J(a_{i},b_{i})\}_{i=1}^{L}$ of $A$ such that
  \[\left\{ \begin{aligned} &J(a_{i},b_{i})\cap A \not= \emptyset, \\
                            &d_{X}(a_{i},b_{i})<D(k). \\
             \end{aligned} 
             \right.
             \]
Since $f$ preserves the causal relation and by the choice of $D(k)$, we have 
\[f(J(a_{i},b_{i})) \subseteq J(f(a_{i}),f(b_{i})),\]
and 
\[d_{Y}(f(a_{i}),f(b_{i}))< \frac{1}{k}.\]
Therefore, we have a covering, $\{J(f(a_{i}),f(b_{i}))\}_{i=1}^{L}$, of $f(A)$ such that 
 \[\left\{ \begin{aligned}  &J(f(a_{i}),f(b_{i}))\cap f(A) \not= \emptyset, \\
                            &d_{Y}(f(a_{i}),f(b_{i}))<\frac{1}{k}. \\
             \end{aligned} 
             \right.
             \]
Furthermore, for $f$ is timelike $\lambda$-Lipschitz, we have
 \[\begin{aligned}
 \mathcal{W}_{\frac{1}{k}}^{N}(f(A))\le \sum_{i=1}^{L} \rho_{N}(J(f(a_{i}),f(b_{i})))
                                     &=\sum_{i=1}^{L} \omega_{N}\tau_{Y}(f(a_{i}),f(b_{i}))^{N}\\
                                     &\le \sum_{i=1}^{L} \omega_{N}\lambda^{N}\tau_{X}(a_{i},b_{i})^{N}.
  \end{aligned}\]
It implies that $\mathcal{W}_{\frac{1}{k}}^{N}(f(A))\le \lambda^{N}\mathcal{W}_{D(k)}^{N}(A)$, and letting $k\to \infty$, $\mathcal{W}^{N}(f(A))\le \lambda^{N}\mathcal{W}^{N}(A)$ follows. 
\end{proof}

\section{Coincidence of $\mathcal{W}^{N}$ and $\mathcal{V}^{N}$}

In general, a uniform upper bound for the distance between vertices of causal diamonds does not give a uniform upper bound for the diameter of causal diamonds. Therefore, $\mathcal{W}^{N}$ may not coincide with $\mathcal{V}^{N}$, and even $\mathcal{W}^{N}$ may fail to be a Borel measure.

On globally hyperbolic smooth spacetimes, we can construct a distance, named null distance, for which the diameter of a causal diamond is bounded from above by the distance between its vertices. From this property, we can see that $\mathcal{V}^{N}$ and $\mathcal{W}^{N}$ with respect to the null distance coincide. Moreover, in smooth spacetimes, there exists a definite null distance that is locally bi-Lipschitz to any Riemannian distance, and we can see that $\mathcal{W}^{N}$ with respect to such null distance coincides with the volume measure with respect to the Lorentzian metric $\mathrm{vol}^{g}$. 

On continuous spacetimes, it is not known if there exists such distance as mentioned above. However, we can construct cylindrical neighborhoods for each point on strongly causal continuous spacetimes \cite[Lemma 4.2]{Mc;Sa}, and we show that $\mathcal{V}^{N} = \mathcal{W}^{N}$ on compact continuous spacetimes with cylindrical neighborhoods. Moreover, we will see that $\mathcal{V}^{N}=\mathcal{W}^{N}$ on Lorentzian warped products.

\subsection{Null distance and $\mathcal{W}^{N}=\mathcal{V}^{N}$ on Lorentzian pre-length spaces}
In \cite{Sor;Ve}, Sormani and Vega introduced a pseudo-metric on smooth spacetimes, named null distance, which stems from generalized time function. Null distances encode both the topological and causal structures on smooth spacetimes. The notion of null distance was extended to Lorentzian pre-length spaces in $\cite{Kun;Ste;NULLd}$. In this subsection, we recall the definition and basic properties of null distance on Lorentzian pre-length spaces based on $\cite{Kun;Ste;NULLd}$, and show that on a Lorentzian pre-length space with a definite null distance, $\mathcal{W}^{N}$ coincides with $\mathcal{V}^{N}$ with respect to the null distance.

\begin{Definition}[Null distance]\label{Def Null distance}
Let $X$ be a Lorentzian pre-length space. We say that $\rho:X \to \mathbb{R}$ is a \emph{generalized time function} on $X$ when $\rho$ is strictly increasing along any non-constant future-directed causal curve. We call a generalized time function $\rho$ a \emph{time function} if $\rho$ is continuous. Let $\rho$ be a generalized time function and $\beta:[a, b]\to X$ be a \emph{piecewise causal curve} from $x=\beta(a)$ to $y=\beta(b)$, i.e. there exists a partition $\{t_{i}\}_{i=1}^{L}$ of $[a,b]$ such that $\beta \vert_{[t_{i}, t_{i+1}]}$ is future or past causal for \textrm{$i$=1, 2, \dots, $L-1$}. Then we define the \emph{null length of $\beta$}, denoted by $\hat{L}_{\rho}(\beta)$ as
\[\hat{L}_{\rho}(\beta)= \sum_{i=1}^{L-1}{\mid}\rho(\beta(t_{i+1}))-\rho(\beta(t_{i})){\mid}.\]
Furthermore, we define \emph{null distance} between $x$ and $y$ as
\[\hat{d}_{\rho}(x,y)=\inf\{\hat{L}_{\rho}(\beta)\;|\;\beta \;\textrm{is a piecewise causal curve from $x$ to $y$}\}.\]
\end{Definition}

\quad Through following Definition \ref{Def scc LpLS} and Lemma \ref{Lem existence piecewise causal cueve}, we can see a sufficient condition for that a Lorentzian pre-length space has null distance. 
\begin{Definition}[\cite{Kun;Ste;NULLd} Definition 3.4]\label{Def scc LpLS}
Let $X$ be a Lorentzian pre-length space. We say that $X$ is $\emph{sufficiently causally connected (scc)}$ if $X$ is path connected, causally path connected, and for any $p\in X$ there exists a timelike curve including $p$. 
\end{Definition}

\begin{Lemma}[\cite{Kun;Ste;NULLd} Lemmas 3.5, 3.7]\label{Lem existence piecewise causal cueve}
Let $X$ be a scc Lorentzian pre-length space with a generalized time function $\rho$. Then, for any $x,y\in X$, there exists a piecewise causal curve connecting $x$ and $y$. In particular, we can define the null distance $\hat{d}_{\rho}$ on $X\times X$. Moreover, $\hat{d}_{\rho}$ is a finite pseudo-metric on $X$.

\end{Lemma}

We can see sufficient conditions for that $d_{\rho}$ is continuous and that topologies induced by $d$ and $d_{\rho}$ coincide through following Lemma \ref{Lem d_rho continuous} and Proposition \ref{Pro d_rho topology}.
\begin{Lemma}[\cite{Kun;Ste;NULLd} Proposition 3.9]\label{Lem d_rho continuous}
$\hat{d}_{\rho}$ is continuous if and only if $\rho$ is continuous on $X$.
\end{Lemma}

 \begin{Proposition}[\cite{Kun;Ste;NULLd} Proposition 3.10]\label{Pro d_rho topology}
Let $X$ be a scc Lorentzian pre-length space with a generalized time function $\rho$ and suppose that $X$ is locally compact. Then, the following are equivalent:

  \begin{itemize}
  
  \item[(i)] $\hat{d}_{\rho}$ induces the same topology on $X$ as that of $d$,
  \item[(ii)] $\rho$ is continuous and $\hat{d}_{\rho}$ is definite.
 
  \end{itemize}
\end{Proposition}

In Definition \ref{Def Locally anti-Lipschitz} and Proposition \ref{Pro d_rho definite}, we get a necessary and sufficient condition that a null distance $d_{\rho}$ is definite on locally compact scc Lorentzian pre-length spaces.
\begin{Definition}\label{Def Locally anti-Lipschitz}
Let $X$ be a Lorentzian pre-length space and $U \subseteq X$. Then, we say that $f:X \to \mathbb{R}$ is \emph{anti-Lipschitz on $U$} if there exists a distance function $d_{U}$ on $U$ such that
\[f(y)-f(x)\ge d_{U}(x,y) \textrm{ for all $x,y\in U$ with $x\le y$}.\]
We call $f$ a \emph{locally anti-Lipschitz function} on $X$ if for any $x\in X$, there exists a neighborhood $U_{x}$ of $x$ where $f$ is anti-Lipschitz.
\end{Definition}

\begin{Proposition}[\cite{Kun;Ste;NULLd} Proposition 3.12]\label{Pro d_rho definite}
Let $\rho$ be a generalized time function on a locally compact scc Lorentzian pre-length space $X$. Then, $\hat{d}_{\rho}$ is definite if and only if $\rho$ is a locally anti-Lipschitz function on X. 
\end{Proposition}

Finally, to get $\mathcal{V}^{N}=\mathcal{W}^{N}$ with respect to null distances, we will see that the distance of vertices of a causal diamond gives an upper bound of the diameter of the causal diamond.

\begin{Lemma}\label{Lem diam for d_rho}
Let $X$ be a scc Lorentzian pre-length space, $\rho$ be a generalized time function on $X$, and $p,q\in X$ with $p\le q$. Then we have
\[\mathrm{diam}_{{\hat{d}_{\rho}}}(J(x,y)) \le 2\hat{d}_{\rho}(x,y).\]
\end{Lemma}

\begin{proof}
  We can get the result directly from $(ii)$ and $(iv)$ of Proposition 3.8 in $\cite{Kun;Ste;NULLd}$.
\end{proof}

The following theorem is a direct implication of Lemma \ref{Lem diam for d_rho}.

\begin{Theorem}\label{Th V^n=W^n null distance}
Let $X$ be a locally compact scc Lorentzian pre-length space with a generalized time function $\rho$ and $\hat{d}_{\rho}$ be the null distance with respect to $\rho$. Then, $\mathcal{W}^{N}$ coincides with $\mathcal{V}^{N}$ with respect to $\hat{d}_{\rho}$.

\end{Theorem}

\subsection{Null distance and $\mathrm{vol}^{g}=\mathcal{W}^{N}$ on globally hyperbolic smooth spacetimes}

We call $X$ a smooth spacetime if $X$ is a connected $N$-manifold equipped with a time oriented and $C^{k}$-differentiable Lorentzian metric $g$, where $k \in \mathbb{N} \cup \{\infty\}$. Any smooth spacetime is a scc Lorentzian pre-length space. In \cite{Ger}, \cite{Ber;San}, and \cite{Ber;San2}, splitting theorems and the existence of time functions on smooth spacetimes are studied. According to these studies, globally hyperbolic smooth spacetimes have time functions and definite null distances. Moreover in \cite[Theorem 1.7]{Bur;Gar}, it is shown that on globally hyperbolic spacetimes, we have a null distance which is locally bi-Lipschitz to any given Riemannian distance. 

The following theorem stems from Lemma 2.6 and Theorem 4.2 in $\cite{Bur;Gar}$.
\begin{Theorem}\label{Th d_h bi-Lipschitz d_rho}
Let $M$ be a globally hyperbolic smooth spacetime and $h$ be any Riemannian metric on $M$. We denote the Riemannian distance with respect to $h$ by $d_{h}$. Then, there exists a time function $\rho$ on $M$, and the null distance $\hat{d}_{\rho}$ is definite. Moreover, for any $p\in M$, there exists a neighborhood $U_{p}$ of $p$ and a constant $C\ge1$ such that
\[\frac{1}{C}d_{h}(x,y) \le \hat{d}_{\rho}(x,y) \le Cd_{h}(x,y)\]
holds for all $x,y\in U_{p}$.
\end{Theorem}

\begin{Theorem}\label{Th Volg=W^n}
Let $M$ be an $N$-dimensional globally hyperbolic smooth spacetime, and $\rho$ and $\hat{d}_{\rho}$ be as in Theorem 4.10. Then, $\mathcal{V}^{N}_{d_{h}}=\mathcal{V}^{N}_{\hat{d}_{\rho}}$, and in particular $\mathcal{W}^{N}_{\hat{d}_{\rho}}=\mathrm{vol}^{g}$ on $M$. Here, $\mathcal{V}^{N}_{d_{h}}$, $\mathcal{V}^{N}_{\hat{d}_{\rho}}$ ,and $\mathcal{W}^{N}_{\hat{d}_{\rho}}$ are with respect to given Riemannian distance $d_{h}$ and the null distance $\hat{d}_{\rho}$ on $M$, respectively. 
\end{Theorem}

\begin{proof}
First, we show that $\mathcal{V}^{N}_{\hat{d}_{\rho}}=\mathcal{V}^{N}_{d_{h}}$. Since $\mathcal{V}^{N}_{\hat{d}_{\rho}}$ is a Borel measure with respect to the topology induced from $d_{h}$, the coincidence of $\mathcal{V}^{N}_{\hat{d}_{\rho}}$ and $\mathcal{V}^{N}_{d_{h}}$ is equivalent to
\[\mathcal{V}^{N}_{\hat{d}_{\rho}}(A)=\mathcal{V}^{N}_{d_{h}}(A) \textrm{\quad for all compact subsets $A\subseteq X$}.\]

Let $A\subseteq X$ be a compact set. Take a finite cover $\{B_{i}\}_{i=1}^{m}$ of $A$ and a sequence $\{C_{i}\}^{m}_{i=1}$ with $C_{i} \ge1$ such that for all $x,y\in B_{i}$
\[\frac{1}{C_{i}}d_{h}(x,y) \le \hat{d}_{\rho}(x,y) \le C_{i}d_{h}(x,y)\]
holds. Let $\lambda$ be the Lebesgue number of $\{B_{i}\}_{i=1}^{m}$. We set $C\coloneq \displaystyle\max_{i=1,2,\dots, m}C_{i}$. Take any $\delta \in (0,\lambda)$ and any countable cover $\{J_{k}\}$ of $A$ such that $J_{k}$ is a causal diamond and $\mathrm{diam}_{{d_{h}}}(J_{k}) \le \delta$. Then, we can assume $J_{k}\subset B_{i}$ for some $i$, and hence $\mathrm{diam}_{{\hat{d}_{\rho}}}(J_{k}) \le \delta C$. Therefore, we have $\mathcal{V}^{N}_{\hat{d}_{\rho},\,\delta C}(A) \le \mathcal{V}^{N}_{d_{h}, \,\delta}(A),$ and with $\delta \to 0$, $\mathcal{V}^{N}_{\hat{d}_{\rho}}(A) \le \mathcal{V}^{N}_{d_{h}}(A)$ follows. Similarly, we have $\mathcal{V}^{N}_{d_{h},\delta C}(A) \le \mathcal{V}^{N}_{\hat{d}_{\rho}, \,\delta}(A)$, and with $\delta \to 0$, $\mathcal{V}^{N}_{d_{h}}(A) \le \mathcal{V}^{N}_{\hat{d}_{\rho}}(A)$ follows, and we have $\mathcal{V}^{N}_{d_{h}}(A) = \mathcal{V}^{N}_{\hat{d}_{\rho}}(A)$.

Next, $\mathcal{W}^{N}_{\hat{d}_{\rho}} = \mathcal{V}^{N}_{\hat{d}_{\rho}}$ follows immediately from Theorem \ref{Th V^n=W^n null distance}. Therefore, we have $\mathrm{vol}^{g}=\mathcal{W}^{N}_{\hat{d}_{\rho}}$ by Theorem \ref{Th Volg=V^n}.
\end{proof}

\subsection{$\mathcal{W}^{N}=\mathcal{V}^{N}$ on continuous spacetimes}
For continuous spacetimes, it is not known if any globally hyperbolic spacetime possesses definite null distances. In this subsection, we show the coincidence of $\mathcal{V}^{N}$ and $\mathcal{W}^{N}$ on continuous compact spacetimes through cylindrical neighborhoods, which is introduced and investigated in \cite{Chr;Gra;C0}, \cite{Lin;C0}, and \cite[Lemma 4.2]{Mc;Sa}.

\begin{Lemma}[Cylindrical neighborhood]\label{Lem cylindrical neighborhood}
Let $C>1$ and $M$ be an $n$-dimensional spacetime with a continuous metric $g$. Then, for each $x$ $\in X$, we can take a pair of neighborhoods $(W,W')$ of $x$ satisfying the following
  \begin{description}

    \item{(1)} $W$ is an open, connected, and relatively compact coordinate chart such that 
       \[\eta_{C^{-1}} \prec g \prec \eta_{C} \text{ on } W,\]
    where $\eta_{C}$ is the scaled Minkowski metric on $\mathbb{R}^{n}$ i.e.
       \[\eta_{C}:=-C^{2}(dt)^{2}+(dx_{1})^{2}+\cdots+(dx_{n-1})^{2}.\]
      Here, we define $g \prec h$ if for any $\upsilon\neq0$, $g(\upsilon,\upsilon)\le0$ implies $h(\upsilon,\upsilon)<0$.
    \item{(2)} $W=(-B,B)\times Z$ with $B>0$ and $Z\subseteq \mathbb{R}^{n-1}$  where $x$ has the coordinates $(0,0)$.

    \item{(3)} $\partial_{t}=\partial_{x_{0}}$ is uniformly timelike on $W$.

    \item{(4)} There is a neighborhood $W'\subseteq (-b,b)\times V \subseteq W$ of $x$ such that $W'$ is causally convex in $W$ i.e. $J(p,q)\subseteq W$ for all $p,q\in W'$.

    \item{(5)} We can take $W$ arbitrarily small and globally hyperbolic.

  \end{description}

\end{Lemma}

We call $(W,W')$ a \emph{cylindrical neighborhood} of $x$.

\begin{Theorem}\label{Th W^n=v^n continuous spacetime}Let $(M,g)$ be an $N$-dimensional compact and continuous strongly causal spacetime. Then, $\mathcal{W}^{N}=\mathcal{V}^{N}$. Here, $\mathcal{W}^{N}$ and $\mathcal{V}^{N}$ are with respect to any given complete Riemannian distance.
\end{Theorem}

\begin{proof}

Let $A\subseteq X$, $h$ be a Riemannian metric on $X$, and $d_{h}$ be the Riemannian distance of $h$. We can easily see that $\mathcal{W}^{N}(A)\le\mathcal{V}^{N}(A)$ from their definitions. To see the reverse inequality, we will show that $\mathrm{diam}_{d_{h}}(J(x,y))\to 0$ uniformly as $d_{h}(x,y)\to 0$ on $X$. Take a finite family of cylindrical neighborhoods $\{(W_{i},W'_{i})\}_{i=1}^{L}$ for $C=2$ such that $\{W'_{i}\}_{i=1}^{L}$ is a finite covering of $A$. Let $D_{i}$ be the Euclidean distance on $W_i =I_i \times Z_i$. Then, for each $i=1,2, \dots, L$, we can take $C_{i}\ge1$ such that $\frac{1}{C_{i}}d_{h}(x,y)\le D_{i}(x, y)\le C_{i}d_{h}(x,y)$ for $x,y\in W$. Let $C'\coloneq \displaystyle\max_{i=1,2,\dots, L}C_{i}$. Let $\delta$ be the Lebesgue number for the covering $\{W'_{i}\}_{i=1}^{L}$ of $A$. Then, take any causal diamond $J(p,q)$ such that $d_{h}(p,q) \le \epsilon \le \delta$ and we find that $J(p,q)\subset W_{i}=:I_{i}\times Z_{i}$ for some $i\in \{1,2,\ldots,L\}$. Let $d_{i}$ be the restriction of $D_{i}$ to $Z_{i}\subset \mathbb{R}^{N-1}$. Write $p=(S,x_{1})$ and $q=(T,x_{2})$ in $W_{i}$. As $g \prec \eta_{2}$ in $W_{i}$, we can see that
\[J(p,q)\subset [S,T]\times \bigl{\{}B^{D_{i}}_{x_{1}}\bigl{(}d_{i}(x_{1},x_{2})+2(T-S)\bigr{)}\cap Z_{i}\bigr{\}},\]
and by $D_{i}(p,q)=\sqrt{(T-S)^{2}+d_{i}(x_{1},x_{2})^{2}}\le C'\epsilon$, 
\[\begin{aligned}
   \mathrm{diam}_{{D_{i}}}(J(p,q)) &\le \sqrt{\{d_{i}(x_{1},x_{2})+2(T-S)\}^{2}+(T-S)^{2}}\\
                  &\le \sqrt{(C'\epsilon+2C'\epsilon)^{2}+C'^{2}\epsilon^{2}}\\
                  &= \sqrt{10}C'\epsilon.
\end{aligned}\]
Then, $\mathrm{diam}_{{d_{h}}}(J(p,q))\le\sqrt{10}C'^{2}\epsilon$ follows. Therefore, we have
\[\mathcal{V}_{\sqrt{10}C'^{2}\epsilon}^{N}(A) \le \mathcal{W}_{\epsilon}^{N}(A),\]
for $0<\epsilon<\delta$, and letting $\epsilon \to 0$ $\mathcal{V}^{N}(A)\le \mathcal{W}^{N}(A)$ follows.
\end{proof}

\subsection{$\mathcal{W}^{N}=\mathcal{V}^{N}$ on Lorentzian warped products}
In this subsection, we will summarize causal structures on Lorentzian warped products (recall Example \ref{Ex Lorentzian warped product}), and show that $\mathcal{V}^{N}$ coincides with $\mathcal{W}^{N}$ on them.

First, we will see causality conditions and the condition of the curvature bounded from below (above). 
\begin{Theorem}[\cite{Ale;Gra;Kun;Sa;cone} Theorem 4.8]\label{Th warped product as LpLS}
Let $(I\times _{f}X, D, \le, \ll, \tau)$ be a Lorentzian warped product. Assume that $(X, d)$ is a locally compact length space. Then $(I\times _{f}X, D, \le, \ll, \tau)$ is strongly causal.
\end{Theorem}

\begin{Theorem}[\cite{Ale;Gra;Kun;Sa;cone} Proposition 4.10]\label{Th Warped product globally hyperbolic}
Let $(I\times _{f}X, D, \le, \ll, \tau)$ be a Lorentzian warped product. Assume that $(X, d)$ is a proper geodesic length space. Then, $(I\times _{f}X, D, \le, \ll, \tau)$ is globally hyperbolic.
\end{Theorem}

\begin{Corollary}[\cite{Ale;Gra;Kun;Sa;cone} Corollary 4.11]
Let $(I\times _{f}X, D, \le, \ll, \tau)$ be a Lorentzian warped product. Assume that $(X, d)$ is a locally compact, complete length space. Then, $(I\times _{f}X, D, \le, \ll, \tau)$ is globally hyperbolic.
\end{Corollary}

\begin{Theorem}[\cite{Ale;Gra;Kun;Sa;cone} Theorem 5.3]\label{Th warped product timelike curvature bound1}
Let $K,K'\in\mathbb{R}$  and let $(X,d)$ be a geodesic length space with curvature bounded from below (above) by $K$. Then, $Y=I\times _{f}X$ has timelike curvature bounded from below (above) by $K'$ if $I\times _{f}M_{K}$ has timelike curvature bounded from below (above) by $K'$. 
  
\end{Theorem}

\begin{Theorem}[\cite{Ale;Gra;Kun;Sa;cone} Theorem 5.7]\label{Th warped product timelike curvature bound2}
If $X$ is a geodesic length space, $I\times _{f}X$ has timelike curvature bound below (above) by $K'$, and $I\times _{f}\mathbb{M}^{2}(K)$ has timelike curvature bound above (below) by $K'$, then $X$ has curvature bound below (above) by $K$.
\end{Theorem}

Next, to get an upper bound of the diameter of causal diamond in Lorentzian warped products, we see that a causal diamond $J((t_{0},\bar{p}),(t_{1},\bar{q}))$ in $Y=I\times _{f}X$ is included in a set whose slice at $t\in[t_{0},t_{1}]$ is the intersection of balls with centers $\bar{p}$ and $\bar{q}$.
\begin{Lemma}[\cite{Ale;Gra;Kun;Sa;cone} Lemma 4.1]\label{Lem causal diamond in warped product}
Let $(Y\coloneq I \times_{f}X, D, \le, \ll, \tau)$ be a Lorentzian warped product, and $p=(t_{0}, \bar{p}), q=(t_{1}, \bar{q})\in Y$. Then, for the causal diamond $J(p, q)=J^{+}(p)\cap J^{-}(q)$, we have
\[J(p, q) \subseteq \{(t, \bar{r})\in Y|t_{0}\le t\le t_{1}, \bar{r}\in \bar{B}^{d}_{\frac{t-t_{0}}{m_{t_{0},t}}}(\bar{p}) \cap \bar{B}^{d}_{\frac{t_{1}-t}{m_{t,t_{1}}}}(\bar{q}) \},\]
where $m_{t_{0}, t_{1}} \coloneq \displaystyle \min_{t_{0}\le t\le t_{1}}\{f(t)\} > 0$, and $\bar{B}^{d}_{\delta}(q) \coloneq \{p\in X\;| \;d(p,q)\le \delta\}$ (see \cite[Definition 3.9]{Ale;Gra;Kun;Sa;cone}).
\end{Lemma}

\begin{Theorem}\label{Th W^n=V^n in warped product}
Let $(I\times _{f}X, D, \le, \ll, \tau)$ be a Lorentzian warped product, and assume that $I \subseteq \mathbb{R}$ is an open interval and that $(X, d)$ is a proper geodesic length space. Then, $\mathcal{W}^{N}$ coincides with $\mathcal{V}^{N}$. Here, $\mathcal{W}^{N}$ and $\mathcal{V}^{N}$ are with respect to the product metric $D$.
\end{Theorem}

\begin{proof}
$\mathcal{W}^{N}\le \mathcal{V}^{N}$ follows from their definitions. Since $I\times_{f}X$ is locally compact, second countable, and Hausdorff, it is sufficient to prove the reverse inequality that 
\[\mathcal{V}^{N}(A)\le \mathcal{W}^{N}(A) \textrm{ for all compact $A\subseteq I\times _{f}X$}.\]
Let $A\subseteq I\times _{f}X$ be compact. We denote the projection to $I$ from $I\times_{f}X$ by $P_{t}$. Take $[a,b]$ as $P_{t}(A)\subset [a,b] \subset I$ and $\epsilon>0$ such that $[a-\epsilon, b+\epsilon]\subset I$. Let $m \coloneq \displaystyle \min_{a-\epsilon\le t\le b+\epsilon}\{f(t)\} > 0$. By Lemma 4.18, for any $p=(t_{0}, \bar{p}), q=(t_{1}, \bar{q})\in [a-\epsilon, b+\epsilon]\times X$ with $p\le q$, we have
 \[J(p, q) \subseteq \{(t, \bar{r})\in Y|t_{0}\le t\le t_{1}, \bar{r}\in \bar{B}^{d}_{\frac{t_{1}-t_{0}}{m}}(\bar{p}) \cap \bar{B}^{d}_{\frac{t_{1}-t_{0}}{m}}(\bar{q}) \}.\]
Therefore, it follows that
 \[\begin{aligned}
     \mathrm{diam}_{{D}}(J(p, q))&\le \sqrt{\biggl{(}d(\bar{p}, \bar{q}) + 2\frac{t_{1}-t_{0}}{m}\biggr{)}^{2}+(t_{1}-t_{0})^{2} }\\
                                &\le \sqrt{\biggl{(}D(p, q) + \frac{2}{m}D(p,q)\biggr{)}^{2}+D(p,q)^{2} } \\
                                &=D(p,q)\sqrt{\biggl{(}1+\frac{2}{m}\biggr{)}^{2}+1}.
\end{aligned}\]
Take any positive number $\delta\in (0,\epsilon)$ and any covering $\{J(a_{i}, b_{i})\}^{L}_{i=1}$ of $A$ such that
 \begin{description}
   \item{(i)} $D(a_{i}, b_{i}) \le \delta$
   \item{(ii)} $J(a_{i}, b_{i}) \cap A \not=\emptyset$.
 \end{description}
Then, $\mathrm{diam}_{{D}}(J(a_{i}, b_{i}))\le m'D(a_{i}, b_{i})$ holds for all $i=0,1,\dots, L$. Here, $m'\coloneq \sqrt{(1+\frac{2}{m})^{2}+1}$. Therefore we have $\mathcal{V}^{N}(A)\le \mathcal{W}^{N}(A).$
\end{proof}

\section{Examples of timelike Lipschitz maps and causality preserving maps}
In this section we will see some examples of timelike Lipschitz maps and causality preserving maps. In particular, we are interested in behavior of timelike extension maps. In Alexandrov spaces, the uniformity of Hausdorff dimension for every open subsets is shown by the volume comparison by extension maps (see \cite[Theorem 10.6.1]{Bra;Bra;Iva}). We examine timelike extension maps in Lorentzian products with timelike non-positive or non-negative, and will find that such map is either timelike Lipschitz or causality preserving.

\begin{Example}(Timelike Lipschitz maps)\label{Ex Timelike Lipschitz map}
Let $(X, d, \le, \ll, \tau)$ be a geodesic Lorentzian pre-length space with timelike curvature bounded from above by $0$ globally and timelike non-branching, i.e. any timelike geodesic does not have a branching point. Let $p\in X$ and $\lambda \in (0,1)$. For every point $q \in I^{+}(p)$, we take a timelike maximizer $\gamma_{p,q}$ from $p$ to $q$ and define $\lambda q$ as the point on $\gamma_{p,q}$ as such $\tau(p, \lambda q)=\lambda\tau(p,q)$. Let $\lambda I^{+}(p)=\bigl{\{} \lambda q \;|\;q\in I^{+}(p) \bigr{\}}$, and we endow $\lambda I^{+}(p)$ with a distance and causal structures by restriction. Therefore we can set a function $f_{\lambda, p}:\lambda I^{+}(p) \to I^{+}(p)$ by $f_{p}(\lambda q)=q$. Then, for $\lambda q_{1}$ and $\lambda q_{2}$ with $f(\lambda q_{1}) \ll f(\lambda q_{2})$, we have
\[\tau(f(\lambda q_{1}), f(\lambda q_{2}))\le \frac{1}{\lambda}\tau(\lambda q_{1}, \lambda q_{2})\]
thanks to the non-positive curvature \cite[Proposition 6.1]{Be;Kun;Rot}. When $f(\lambda q_{1}) \not\ll f(\lambda q_{2})$, the inequality above holds obviously. Therefore $f_{\lambda, p}:\lambda I^{+}(p) \to I^{+}(p)$ is a timelike $\frac{1}{\lambda}$-Lipschitz map. 
\end{Example}

To generalize this result to Lorentzian pre-length spaces with timelike curvature bounded from above by $k>0$ globally, we need to find a $\lambda$ dependent constant $C(\lambda)$ as follows. Let $\lambda \in (0,1)$, $\tilde{S}_1^2(k)$ be the model space with the constant curvature $k>0$, and $\Delta pqr$ be any timelike geodesic triangle in $\tilde{S}_1^2(k)$ with $p\ll q\ll r$ and sides $\gamma_{p,q}$, $\gamma_{q,r}$, and $\gamma_{p,r}$ of timelike geodesics. We set $\lambda q\in \gamma_{p,q}$ and $\lambda r\in \gamma_{p,r}$ such that $\tau(p,\lambda q)=\lambda\tau(p,q)$ and $\tau(p,\lambda r)=\lambda\tau(p,r)$. Then, when $\lambda q \ll \lambda r$,
\[C(\lambda)\tau(\lambda q, \lambda r)\ge\tau(q, r)\]
holds.

\begin{Example}(Causality preserving maps)\label{Ex Causal preserving map}
Let $(X,d)$ be a geodesic length space. Let $(\mathbb{R}\times X, D, \le, \ll, \tau)$ be the Lorentzian product, i.e. the warped product, $Y=\mathbb{R}\times X_{f}$ with $f\equiv 1$. Assume that $\mathbb{R}\times X$ has timelike curvature bounded from below by $0$. Let $p=(0, x)\in \mathbb{R}\times X$, and $q=(T, y)\in I^{+}(p)$. When $x\not=y$, we can take a geodesic from $p$ to $q$, $\gamma_{p,q}:[0, 1]\to \mathbb{R}\times X$ such as $\gamma_{p,q}(t)=(tT, \Gamma_{x, y}(td(x,y)))$. Here, $\Gamma_{x, y}:[0,d(x, y)]\to X$ is a geodesic in $X$ from $x$ to $y$ with unit speed. When $x=y$, we can take a geodesic from $p$ to $q$, $\gamma_{(p,q)}:[0, 1]\to \mathbb{R}\times X$ as $\gamma_{p,q}(t)=(tT, x)$. 

Let $\lambda\in (0,1)$, and we set $\lambda I^{+}(p)=\{\gamma_{(p,q)}(\lambda)\;|\; q=(T,y)\in I^{+}(p), \;T>0\}$. We denote $\gamma_{(p,q)}(\lambda)$ by $\lambda q$ for $q\in I^{+}(p)$. Recall from Theorem \ref{Th warped product timelike curvature bound1} and \ref{Th warped product timelike curvature bound2} that when $\mathbb{R}\times X$ has timelike curvature bounded from below by $0$, $X$ has curvature bounded from below by $0$. Therefore, $X$ does not have any branching, and we can set a function $g_{\lambda, p}:\lambda I^{+}(p)\to X$ as $g_{\lambda ,p}(\lambda q)=q$. Here, we endow $\lambda I^{+}(p)$ with a distance and causal structures by restriction, $(\lambda I^{+}(p), D \vert_{\lambda I^{+}(p)}, {\le}\vert_{\lambda I^{+}(p)}, {\ll}\vert_{\lambda I^{+}(p)}, \tau\vert_{\lambda I^{+}(p)})$. Then, $g_{\lambda, p}$ is a causality preserving map from $\lambda I^{+}(p)$ to $X$. 

Take $\lambda q_{1},\lambda q_{2}\in \lambda I^{+}(p)$. Write $q_{1}=(T_{1}, y_{1})$ and $q_{2}=(T_{2}. y_{2})$, and we see $\lambda q_{1}=(\lambda T_{1}, \Gamma_{x,y_{1}}(\lambda d(x,y_{1})))$ and $\lambda q_{2}=(\lambda T_{2}, \Gamma_{x,y_{2}}(\lambda d(x,y_{2})))$.  Assume $\lambda q_{1}\le \lambda q_{2}$ i.e. $\lambda T_{2}-\lambda T_{1} \ge d(\Gamma_{p,y_{1}}(\lambda d(x,y_{1})),\Gamma_{p,y_{2}}(\lambda d(x,y_{2})))$. Then, we have
\[\begin{aligned}
  T_{2}-T_{1} &\ge \frac{d(\Gamma_{p,y_{1}}(\lambda d(x,y_{1})),\Gamma_{p,y_{2}}(\lambda d(x,y_{2})))}{\lambda} \\
              &\ge d(\Gamma_{p,y_{1}}(d(x,y_{1})),\Gamma_{p,y_{2}}(d(x,y_{2}))) \\
              &=d(q_{1}, q_{2}).
\end{aligned}\]
It implies $q_{1}\le q_{2}$. The second inequality above holds since $X$ has curvature bounded from below by $0$.

\end{Example}

In the case of Lorentzian products with timelike curvature bounded from below by $k<0$, with a constant $R>\max{(d(p,y_{1}), d(p,y_{2}))}$, we have 
\[d(\Gamma_{p,y_{1}}(\lambda d(p,y_{1})),\Gamma_{p,y_{2}}(\lambda d(p,y_{2})))\ge C(-k,\lambda, R)d(\Gamma_{p,y_{1}}(d(p,y_{1}))),\Gamma_{p,y_{2}}(d(p,y_{2}))),\]
where $C(-k, \lambda, R)=\frac{\sinh(-k\lambda R)}{\sinh(-kR)}$. Then, since $\frac{C(-k,\lambda, R)}{\lambda}<1\;(\lambda\in(0,1))$, we can see that $g_{\lambda,p}$ does not necessarily preserve causality.

\section{Outlook} 
We mention further researches related to this work. As we saw in Section 5, it seems not easy to construct a timelike Lipschitz map which is also causality preserving. Thus, we would need some improvements of our volume comparison inequalities. For example, it is desirable to find volume comparison inequality for merely for timelike Lipschitz maps. Another direction of further researches is to apply the volume comparison inequalities to the theory of convergence and collapsing of spacetimes. Moreover, in view of constructing local charts on Lorentzian pre-length spaces, we could work on evaluating the timelike Hausdorff dimension of the set of singular points. The properties of timelike Hausdorff measure are also our interest, for example the density of timelike Hausdorff measure and area-coarea inequalities for timelike Hausdorff measure on Lorentzian pre-length spaces.


\begin{thebibliography}{1}


\bibitem{Ale;Gra;Kun;Sa;cone}
S. B. Alexander, M. Graf, M. Kunzinger, and C. S\"{a}mann.
\newblock {\em Generalized cones as Lorentzian length spaces: Causality, curvature, and singularity theorems}.
\newblock Comm. Anal. Geom. {\bf 31} (2023), no.~6, 1469--1528; MR4785565.



\bibitem{Be;Ku;Oha;Ro}
T. Beran, M. Kunzinger, A. Ohanyan, and F. Rott.
\newblock {\em The equivalence of smooth and synthetic notions of timelike sectional curvature bounds}.
\newblock Proc. Amer. Math. Soc. {\bf 153} (2025), no.~2, 783--797; MR4852799.


\bibitem{Be;Kun;Rot}
T. Beran, M. Kunzinger, F. Rott.
\newblock {\em On curvature bounds in Lorentzian length spaces}.
\newblock J. Lond. Math. Soc. (2) {\bf 110} (2024), no.~2, Paper No. e12971, 41 pp.; MR4781260.


\bibitem{Be;Oha;Ro;So;splliting}
T. Beran, A. Ohanyan, F. Rott, and D. A. Solis.
\newblock {\em The splitting theorem for globally hyperbolic Lorentzian length spaces with non-negative timelike curvature}.
\newblock Lett. Math. Phys. {\bf 113} (2023), no.~2, Paper No. 48, 47 pp.; MR4579262.


\bibitem{Be;Sa;angle}
T. Beran and C. S\"{a}mann.
\newblock {\em Hyperbolic angles in Lorentzian length spaces and timelike curvature bounds}.
\newblock  J. Lond. Math. Soc. (2) {\bf 107} (2023), no.~5, 1823--1880; MR4585303.


\bibitem{Ber;San2}
A. N. Bernal and M. S\'{a}nchez.
\newblock{\em On Smooth Cauchy Hypersurfaces and Geroch's Splitting Theorem}.
\newblock Comm. Math. Phys. {\bf 243} (2003), no.~3, 461--470; MR2029362.


\bibitem{Ber;San}
A. N. Bernal and M. S\'{a}nchez.
\newblock{\em Smoothness of Time Functions and the Metric Splitting of Globally Hyperbolic Spacetimes}.
\newblock Comm. Math. Phys. {\bf 257} (2005), no.~1, 43--50; MR2163568.

\bibitem{Br}
M. Braun.
\newblock {\em Rényi's entropy on Lorentzian spaces. Timelike curvature-dimension conditions}.
\newblock Trans. Amer. Math. Soc. {\bf 377} (2024), no.~5, 3529--3576; MR4744787.
J. Math. Pures Appl {\bf 177} (2023), Pages 46-128, ISSN 0021-7824.

\bibitem{Br;Oh}
M. Braun and S. Ohta.
\newblock {\em Optimal transport and timelike lower Ricci curvature bounds on Finsler spacetimes}.
\newblock Trans. Amer. Math. Soc. {\bf 377} (2024), no.~5, 3529--3576; MR4744787.


\bibitem{Bra;Bra;Iva}
D. Burago, Y. Burago, and S. Ivanov.
\newblock {\em A Course in Metric Geometry}.
\newblock volume 33 of Graduate Studies in Mathematics. American Mathematical Society, Providence, RI, (2001).

\bibitem{Bur;Gar}
A. Burtscher and L. Garc\'{i}a-Heveling.
\newblock{\em Global Hyperbolicity through the Eyes of the Null Distance}.
\newblock Comm. Math. Phys. {\bf 405} (2024), no.~4, Paper No. 90, 35 pp.; MR4719972.



\bibitem{Cav;Mon}
F. Cavalletti and A. Mondino.
\newblock {\em Optimal transport in Lorentzian synthetic spaces, synthetic
timelike Ricci curvature lower bounds and applications}.
\newblock Camb. J. Math. {\bf 12} (2024), no.~2, 417--534; MR4779676.


\bibitem{Chr;Gra;C0}
P. T. Chru\'{s}ciel and  J. D. E. Grant.
\newblock {\em On Lorentzian causality with continuous metrics}.
\newblock Classical Quantum Gravity {\bf 29} (2012), no.~14, 145001, 32 pp.; MR2949547.


\bibitem{Gal;Gra;Li}
G. J. Galloway, M. Graf, and E. Ling.
\newblock {\em A Conformal Infinity Approach to Asymptotically $AdS_{2}\times S^{n-1}$ Spacetimes}.
\newblock Ann. Henri Poincar\'e{} {\bf 21} (2020), no.~12, 4073--4095; MR4172943.


\bibitem{Gar;Sou}
L. Garc\'{i}a-Heveling and E. Soultanis.
\newblock {\em Causal bubbles in globally hyperbolic spacetime}.
\newblock Gen. Relativity Gravitation {\bf 54} (2022), no.~12, Paper No. 155, 7 pp.; MR4517255.


\bibitem{Ger}
R. Geroch.
\newblock{\em Domain of dependence}.
\newblock J. Math. Phys. 1 February 1970; 11 (2): 437–449.


\bibitem{Gra;TC1}
M. Graf.
\newblock {\em Singularity theorems for $C^{1}$-Lorentzan metric}.
\newblock Comm. Math. Phys. {\bf 378} (2020), no.~2, 1417--1450; MR4134950.


\bibitem{Gra;Gran;Kun;Ste;HPC11}
M. Graf, J. D. E. Grant, M. Kunzinger, and R. Steinbauer.
\newblock {\em The Hawking–Penrose singularity theorem for $C^{1,1}$-Lorentzian}.
\newblock Comm. Math. Phys. {\bf 360} (2018), no.~3, 1009--1042; MR3803816.


\bibitem{Gra;Kun;Sa;Ste;future}
J. D. E. Grant, M. Kunzinger, C. S\"{a}mann, and R. Steinbeur.
\newblock {\em The future is not always open}.
\newblock Lett. Math. Phys. {\bf 110} (2020), no.~1, 83--103; MR4047145.



\bibitem{Kun;Oha;Sch;Ste;HPC1}
M. Kunzinger, A. Ohanyan, B. Schinnerl, and R. Steinbauer.
\newblock {\em The Hawking–Penrose Singularity Theorem for $C^1$-Lorentzian Metrics}.
\newblock Comm. Math. Phys. {\bf 391} (2022), no.~3, 1143--1179; MR4405569.


\bibitem{Kun;Sa;pre-length}
M. Kunzinger and C. S\"{a}mann.
\newblock {\em Lorentzian length spaces}.
\newblock Ann. Global Anal. Geom. {\bf 54} (2018), no.~3, 399--447; MR3867652.


\bibitem{Kun;Ste;NULLd}
M. Kunzinger and R. Steinbauer.
\newblock {\em Null Distance and Convergence of Lorentzian Length Spaces}.
\newblock Ann. Henri Poincar\'e{} {\bf 23} (2022), no.~12, 4319--4342; MR4512238.


\bibitem{Kun;Ste;Vic;HC11}
M. Kunzinger, R. Steinbauer, M. Stojko\'{v}ic, and J. A. Vickers.
\newblock {\em Hawking's singularity theorem for $C^{1.1}$-metrics}.
\newblock Classical Quantum Gravity {\bf 32} (2015), no.~7, 075012, 19 pp.; MR3322151.


\bibitem{Kun;Ste;Vic;PC11}
M. Kunzinger, R. Steinbauer, J. A. Vickers.
\newblock {\em The Penrose singularity theorem in regularity $C^{1.1}$}.
\newblock Classical Quantum Gravity {\bf 32} (2015), no.~15, 155010, 12 pp.; MR3368951.



\bibitem{Lin;C0}
E. Ling.
\newblock {\em  Aspects of $C^0$ causal theory}.
\newblock Gen. Relativity Gravitation {\bf 52} (2020), no.~6, Paper No. 57, 40 pp.; MR4111861.



\bibitem{McC;De}
R. J. McCann.
\newblock {\em Displacement convexity of Boltzmann's entropy characterizes the strong energy condition from general relativity}.
\newblock  Camb. J. Math. {\bf 8} (2020), no.~3, 609--681; MR4192570.


\bibitem{Mc;Sa}
R. J. McCann and C. S\"{a}mann.
\newblock {\em A Lorentzian analog for Hausdorff dimension and measure}.
\newblock Pure Appl. Anal. {\bf 4} (2022), no.~2, 367--400; MR4496090.


\bibitem{Min;cone}
E. Minguzzi.
\newblock {\em Causality theory for closed cone structures with applications}.
\newblock Rev. Math. Phys. {\bf 31} (2019), no.~5, 1930001, 139 pp.; MR3955368.


\bibitem{Min;Suh1}
E. Minguzzi and S. Suhr.
\newblock {\em Lorentzian metric spaces and their Gromov–Hausdorff convergence}.
\newblock Lett. Math. Phys. {\bf 114} (2024), no.~3, Paper No. 73, 63 pp.; MR4752400.


\bibitem{Mon;Su}
A. Mondino and A. Suhr.
\newblock {\em An optimal transport formulation of the Einstein equations of general relativity}.
\newblock  J. Eur. Math. Soc. (JEMS) {\bf 25} (2023), no.~3, 933--994; MR4577957.


\bibitem{Mul}
O. Müller. 
\newblock {\em Lorentzian Gromov–Hausdorff theory and finiteness results}.
\newblock Gen. Relativity Gravitation {\bf 54} (2022), no.~10, Paper No. 117, 17 pp.; MR4493631.



\bibitem{Nol}
J. Noldus.
\newblock {\em  A Lorentzian Gromov-Hausdorff notion of distance}.
\newblock Classical Quantum Gravity {\bf 21} (2004), no.~4, 839--850; MR2036128.



\bibitem{Sor;Ve}
C. Sormani, C. Vega.
\newblock {\em Null distance on a spacetime}.
\newblock  Classical Quantum Gravity {\bf 33} (2016), no.~8, 085001, 29 pp.; MR3476515.


\bibitem{Ste}
R. Steinbauer.
\newblock {\em The singularity theorems of General Relativity and their low regularity extensions}.
\newblock Jahresber. Dtsch. Math.-Ver. {\bf 125} (2023), no.~2, 73--119; MR4594980.

\end{thebibliography}
\end{document}